\documentclass{amsart}
\usepackage{amsfonts,amscd, amssymb}

\newtheorem{theorem}{Theorem}[section]
\newtheorem{lemma}[theorem]{Lemma}
\newtheorem{corollary}[theorem]{Corollary}
\newtheorem{proposition}[theorem]{Proposition}
\theoremstyle{remark}

\theoremstyle{definition}

\numberwithin{equation}{section}
\makeatother

\DeclareMathOperator{\Fdb}{{\mathbb F}}

\newcommand{\acts}{\curvearrowright}

\DeclareMathOperator{\cS}{{\mathcal S}}

\DeclareMathOperator{\Cdb}{{\mathbb C}}
\DeclareMathOperator{\Rdb}{{\mathbb R}}
\DeclareMathOperator{\Zdb}{{\mathbb Z}}
\DeclareMathOperator{\Hdb}{{\mathbb H}}

\DeclareMathOperator{\Tdb}{{\mathbb T}}
\DeclareMathOperator{\Ndb}{{\mathbb N}}

\begin{document}

\title[Real operator spaces and algebras]{Real structure in operator spaces,  injective envelopes and $G$-spaces}
\author{David P. Blecher}
\address{Department of Mathematics, University of Houston, Houston, TX
77204-3008, USA}
\email{dpbleche@central.uh.edu}

\author{Arianna Cecco}
\address{Department of Mathematics, University of Houston, Houston, TX
77204-3008, USA}
\email{ahcecco@central.uh.edu}

\author{Mehrdad Kalantar}
\address{Department of Mathematics, University of Houston, Houston, TX
77204-3008, USA}
\email{kalantar@math.uh.edu} 

\date{3/29/2023} 
\thanks{MK and AC are  supported by the NSF Grant DMS-2155162.  DB  is supported by a Simons Foundation Collaboration Grant.}

\begin{abstract}  We  present some more foundations for a theory of real structure in 
operator spaces and algebras, in particular concerning  the real case of the theory of injectivity, and the injective, ternary, and $C^*$-envelope.
We  consider the interaction between these topics and the complexification.   
We also generalize many of these results to the setting of operator spaces and systems  acted upon by a group. 
  \end{abstract}
  
\maketitle
\section{Introduction}

Ruan initiated the study of real operator spaces in \cite{ROnr,RComp}, and this study was continued in 
\cite{Sharma} and \cite{BT}.     Recently there has been an increased interest in this topic, and in  its links to quantum physics (see e.g.\ \cite{Ch1, Ch2} and references therein).   Indeed real structure occurs naturally in very many areas of mathematics, as is mentioned also for example in the first
paragraphs of \cite{BT} and \cite{Sharma}, or in \cite{Ros}.  
In the present paper we examine some aspects of 
real  and complex structure in operator spaces, establishing the real case of several important aspects 
of the theory of injectivity, and  the injective, ternary, 
and $C^*$-envelope, advancing on Sharma's work in \cite{Sharma}.  We also 
discuss the interaction between these topics and the complexification, and establish some natural relations between injectivity in the real and the complex categories. 
In fact, we undertake much of this study in a more general setting where the given operator space $X$ is equipped with a completely isometric action of a discrete group $G$. Then the results in the usual setting of operator spaces follow from the case where the group $G$ is trivial. 
Recently, the theory of equivariant injective envelopes has found very many striking applications in various problems concerning the structure theory of group $C^*$-algebras 
(see e.g.\ \cite{KK,BKKO,Bry,KSg,KS,KKLRU} and references therein). 
In the light of this, and considering the recent emergence of  interest in real $C^*$-algebras alluded to above, it is natural to investigate $G$-spaces in the real case.

Turning to the structure of our paper, in Section 2 we review the basic theory of 
operator space complexifications, including a short proof of Ruan's uniqueness of a reasonable complexification.
In  Section \ref{chcvn} we give a characterization of commutative real $W^*$-algebras in the most general case.  (We stated this result in \cite{BT}, and it will be used in 
\cite{BReal} to prove a couple of  nice facts about real operator spaces and algebras.) 
We begin   Section \ref{Coen} gently by  considering the relations between  the complexification and  injective  
and $C^*$-envelopes. 
As pointed out in \cite{BWinv} many aspects of the study of real operator spaces or algebras are equivalent to the study of real structure in complex spaces or algebras, in particular 
the features of  conjugate linear completely isometric period 2 automorphisms 
on complex operator spaces or algebras.

  We continue   Section \ref{Coen}   by considering 
real and complex operator spaces and  systems with a $G$-action by a discrete group $G$, and 
generalize much of our earlier theory (which in some sense is the case that $G = \Zdb_2$) to this setting.   
In particular we study  the $G$-$C^*$-envelope and $G$-injective envelope, following Hamana's work 
in the complex case \cite{Hamiecds, Hamiods}.    The second author has undertaken a very systematic investigation of injective  
 envelopes in various categories in  \cite{Cecco,CeccoTh}, and part of this (and several  of  the preliminary results needed for this) is established in this section. 
 For example we prove that the real and complex $G$-injective envelopes of a complex $G$-operator space coincide, and that the $G$-injective envelope of the 
 complexification of a real $G$-operator space $X$ is canonically identified with the complexification of  the $G$-injective envelope of $X$. We also
  prove similar facts for $G$-$C^*$ and $G$-ternary envelopes.

 We remark that Hamana works in a great generality in e.g.\ \cite{Hamiods} that includes all locally compact groups, and indeed Hopf von Neumann algebras, which 
are more general still.   Most of our results in Section 
 \ref{Coen} 
 seem to be true in the real case of Hamana's hugely general setting, by appropriate modification of the same basic proof ideas, and 
using (the real case of) the matching facts 
in \cite{Hamiods}.   However for simplicity we will stick to discrete groups since  the idea of the proofs would be essentially unchanged in the general case.  
An exception to this remark is our results on finite groups.  These seem to be easily generalizable to actions by a compact group, but seem to us to be  difficult to generalize much further than this (see Remark 
below Theorem \ref{Ginj}).   
In Section \ref{fe} we discuss  extending real structure on an algebra to containing von Neumann algebras. 

The reader will need to be familiar with the  basics of operator spaces and von Neumann algebras  
as may be found in early chapters of \cite{BLM,ER,Pau, Pisbk}, and e.g.\  \cite{P}.  
With familiarity with the very basics the reader  will have no problems following Sections 2, \ref{chcvn}, and \ref{fe}. 
Section \ref{Coen} will assume more familiarity with (the usual complex) operator space
theory.  For the injective envelope,  $C^*$-envelope,  and ternary envelope in the complex case we refer to 
  \cite[Sections 4.2--4.4]{BLM}, \cite[Section 6.2]{ER}, \cite[Chapter 15]{Pau}, or the 
papers of Hamana and Ruan referenced there.  Then  \cite{Good, Li,ARU} are texts on 
the theory of real $C^*$-algebras and real $W^*$-algebras.   We recall that a real $C^*$-algebra is a closed real $*$-subalgebra of $B(H)$ for a real Hilbert space $H$.
Abstractly it is a  real Banach $*$-algebra whose  $*$-algebra complexification has a $C^*$-norm making it a complex $C^*$-algebra.
Equivalently, it  is a  real Banach $*$-algebra  satisfying the $C^*$-identity whose selfadjoint elements have real spectrum (or alternatively, 
with $a^* a$ having positive spectrum, or $1+a^*a$ being invertible,  for $a \in A$).   See \cite[Chapter 5]{Li}. 
A space $Z$ in a category will be called {\em injective} if whenever $E \subset F$ as subobjects in the category then any morphism (e.g.\ real linear complete contraction
in the category of real operator spaces) $T : E \to Z$ has  an extending morphism $\tilde{T} : F \to Z$.  E.g.\ $B(H)$ is injective
in the category of real operator spaces, if $H$ is a real Hilbert space \cite{ROnr}. 
 A preliminary study of injective spaces, the injective envelope $I(X)$ and $C^*$-envelope $C^*_e(X)$ in the real case may be found in \cite{Sharma}.   
  Indeed we will 
  be using selected results and notation from the
existing theory of real operator spaces \cite{ROnr,RComp,Sharma,BT}.  Section 
\ref{Coen} will also use basic ideas from Hamana's theory of 
$G$-injective envelopes of $G$-spaces from \cite{Hamiecds}.  

 The letters $H, K$ are reserved for real Hilbert spaces.
 We sometimes write the complex number $i$ as $\iota$ to avoid confusion with matrix subscripting. 
 For us a {\em projection}  in a $*$-algebra 
is always an orthogonal projection (so $p = p^2 = p^*$).    A  normed algebra $A$  is {\em unital} if it has an identity $1$ of norm $1$, 
and a map $T$ 
is unital if $T(1) = 1$.  We say that $A$ is {\em approximately unital} if it has a contractive approximate identity (cai).
 We write $X_+$ for the positive operators (in the usual sense) that happen to
belong to $X$.    If $X$ is a subspace of a (real or complex) $C^*$-algebra $B$ then we write $C^*(X)$ or $C^*_B(X)$ for the
$C^*$-subalgebra of $B$ generated by $X$.  

If $T : X \to Y$ we write $T^{(n)}$ for the canonical `entrywise' amplification taking $M_n(X)$ to $M_n(Y)$.   
The completely bounded norm is $\| T \|_{\rm cb} = \sup_n \, \| T^{(n)} \|$, and $T$ is 
completely  contractive if  $\| T \|_{\rm cb}  \leq 1$. 
A map $T$ is said to be {\em positive} if it takes  positive elements to positive elements, and  {\em 
completely positive} if $T^{(n)}$ is  positive for all $n \in \Ndb$. A UCP map is  unital and completely positive.

A real operator space may either be viewed as a real subspace of $B(H)$ for a real Hilbert space $H$, or abstractly as 
a vector space with a norm $\| \cdot \|_n$ on $M_n(X)$ for each $n \in \Ndb$ satisfying  the conditions of
Ruan's characterization  \cite{ROnr}.    Sometimes the sequence of norms $(\| \cdot \|_n)$ is
called the {\em operator space structure}.  All spaces in the present paper are such operator spaces at the least, but often have more structure.
Then $X_c = X + i X$, this is the complexification that will be discussed in detail in Section 2. 
If $T : X \to Y$ then we write $T_c$  for the complexified map $x + i y \mapsto T(x) + i T(y)$ for $x, y \in X$.

A {\em unital operator space} (resp.\ {\em real operator system}) is a real subspace  (resp.\  selfadjoint subspace) of $B(H)$ for a real Hilbert space $H$,
containing $I_H$.  The  basics of the theory of real operator systems are much the same as in the complex case (see  \cite{RComp,ROnr, Sharma,BT,BReal}, most results
being proved in the same way as the complex case, or simply following from that case by complexification.  One does need to be careful
about a couple of issues though.  Arguments involving states or positive or selfadjoint  elements
often do not work in the real case, as pointed out e.g.\ in \cite{BT}.  
Indeed  there may be not enough, or there may be too many, selfadjoint elements.  
It is shown however in \cite[Lemma 2.3]{BT} that a completely positive map on a real operator system is selfadjoint (i.e.\ $T(x)^* = T(x^*)$) 
and $T_c$ is completely positive.  Also a unital map between real operator systems is completely positive  if and only if it is completely 
contractive.  The one direction of this follows by the just cited lemma.  Conversely
if $T$ is unital and and completely 
contractive then so is $T_c$, so by the complex theory $T_c$ (and hence $T$) is UCP.
A unital completely 
contractive map between real unital operator spaces of course extends to a completely completely 
contractive, hence UCP map on any containing operator system.    A homomorphism between real $C^*$-algebras is contractive if and only if it is a $*$-homomorphism,
and in this case it  is  completely positive and completely contractive \cite[Theorem 2.6]{BT}. 
We remark that  \cite{BReal} includes a systematic discussion of the real variants 
of results in  Chapters 1--4 and 8 of \cite{BLM} (results which are not in the present paper); and checks the behaviour of the complexification 
of very many  standard constructions in the theory, etc. 

An injective envelope $I(X)$ of $X$ in any of our categories is a pair $(E,j)$ consisting of an space $E$ which is injective  in our category
 and completely isometric morphism $j : X \to E$
which is {\em rigid}: that is, $I_E$ is the only completely contractive morphism $u : E \to E$ extending the identity map on $j(X)$.  We will not use 
this but the  injective envelope may also be 
characterized in terms of the `essential' or `envelope' property as in e.g.\ \cite[Lemma 4.2.4]{BLM}.  We sometimes write $I(X)$ as $I_{\Fdb}(X)$, where $\Fdb = \Rdb$ or $\Cdb$ if it is important to distinguish between the real and complex case.  Often we do not do this though,  and leave it to the reader to distinguish what is meant  from the context. 

The injective envelope may be given the structure of a {\em ternary system}, or TRO for short.    We mention  the  definition and the real versions of some basic results on TRO's from e.g.\ \cite{BLM} that we will need 
 in Section 4.    Much of this is already in the real $JB$-triple literature or in e.g.\ \cite{Sharma}.
A  real TRO  is a closed linear subspace $Z \subset B(K, H),$ for real Hilbert spaces $K$ and $H$ , satisfying $Z Z^* Z \subset Z$ .  A {\em ternary  morphism} between TROs is a linear map satisfying $T (x y^* z) = T (x) T( y)^* T(z)$.
If $Z \subset B_{\Rdb}(K,H)$ is a real TRO then (the closures of) 
$Z^* Z$ and $Z Z^*$ are real $C^*$-subalgebras of $B(K)$ and $B(H)$ respectively. 
As in \cite[p.\ 1052]{RComp} we have identifications 
$$Z_c = Z + i Z \subset B(K,H) + i B(K,H) = B(K,H)_c = B_{\Cdb}(K_c,H_c).$$   
Then $Z_c Z_c^* Z_c \subset Z_c$, so that $Z_c$ is a complex TRO. 
The real case of   \cite[Lemma 8.3.2]{BLM} holds: 
 a ternary morphism between real TRO's is completely contractive, and is completely isometric if it is also one-to-one. 
Indeed if $T : Z \to W$ is a real ternary morphism then $T_c$ is a complex ternary morphism. 
So by the complex theory (see e.g.\  \cite[Lemma 8.3.2]{BLM} ) $T_c$ and hence $T$ are completely contractive.  If also $T$ is one-to-one then $T_c$ is one-to-one so 
completely  isometric. 
Hence $T$ is completely isometric.      
 It is proved in e.g.\ \cite{Sharma} that conversely a completely isometric surjection between real TRO's is a ternary morphism (or this follows from the 
 complex case by passing to the complexifications). 
    
\section{Complexifications of real operator spaces}

One of the cornerstones of the study of real operator spaces
is Ruan's unique complexification theorem.   By an {\em operator space complexification} of a real operator space $X$ 
we mean a pair $(X_c, \kappa)$ consisting of a complex operator space $X_c$ and a real linear complete isometry $\kappa : X \to X_c$ 
such that $X_c = \kappa(X) \oplus i \, \kappa(X)$ as a vector space.   For simplicity we usually identify $X$ and $\kappa(X)$ and write $X_c = X + i \, X$.

\medskip

{\bf Remarks.}  1)\  By the closed graph theorem  the projection $P$ onto $i(X)$ is bounded, and so $X_c$ is real isomorphic to a Banach space direct sum of 
$X$ with itself.  Some authors may  also assume that the projection $P$ onto $i(X)$ is contractive (in our case we would want completely contractive here), but for `reasonable' 
complexifications (see below)  this will be automatic. 
Note that $x = P(x) - i P(ix)$ for all $x \in X_c$. 

\smallskip

 2)\ Not every complex operator space is the complexification of a real operator space.  Indeed  on  a complex operator space $X$ there need not 
 exist  any closed subspace $Y$ with $X = Y \oplus iY$.    Dramatic examples of this may be found in e.g.\ \cite{FerR}, which contains infinite dimensional complex Banach spaces 
 which are not the direct sum of any two infinite dimensional real subspaces (indeed the examples there are much more startling).   For a simpler example one may use 
 for example any complex Banach space $X$ which is not complex isomorphic to its complex conjugate $\bar{X}$ (see e.g.\ \cite{Kalton} for an elementary example).    
 If $X = Y + i \, Y$ as above then
 the map $\theta(y_1 + i y_2) = y_1 - i y_2$ is a bicontinuous isomorphism onto $\bar{X}$.    Indeed if $z = y_1 - i y_2$ with $y_1, y_2 \in Y$ then 
 $\| z \|_{\bar{X}} \leq \|  y_1 \| + \|  y_2 \| = \| P(z) \| + \| P(iz) \|$ is dominated by a constant times $\| z \|_X$.  So the map is bicontinuous
 by the open mapping theorem.   This is a contradiction (one may assign any compatible operator space structure to  $X$, such as Min$(X)$). 
 
 The above also shows that complex operator spaces which are real linearly completely isometric, actually need not have any complex linear completely isometric
 surjection between them. For a $C^*$-algebra example of this, if $A$ is a unital C*-algebra not $*$-isomorphic to its opposite C*-algebra $A^\circ$ (see e.g.\ \cite{C}),   consider 
the map $\theta(a) = (a^*)^\circ$.   This is a surjective real linear unital complete isometry $A \to A^\circ$ (even positive 
since if $a = x^2$ for $x = x^*$ then $a^\circ = (x^\circ)^2$).
However by an operator space version of the Kadison-Banach-Stone theorem, if there exists a
 surjective complex linear complete isometry $A \to A^\circ$ then $A$ is $*$-isomorphic to $A^\circ$
 (see e.g.\ \cite[Theorem 4.4]{RComp}).

 \medskip

We say that an operator space complexification $X_c = X + iX$ of an operator space $X$ 
is {\em (completely) reasonable} if  the map 
$\theta_X : x+iy \mapsto x - iy$ is a complete isometry,  for $x, y \in X$.  
Any reasonable complexification of a Banach space possesses a canonical conjugate linear isometric period 2 automorphism. Conversely, given any complex Banach space $X$ with a conjugate linear  isometric period 2 automorphism $\theta$, it is well known (or an exercise) 
  that $X$ is the reasonable complexification of the set of fixed points of $\theta$.  A similar result holds for 
 operator spaces.  A completely reasonable operator space complexification  $X_c$ possesses a canonical conjugate linear  completely isometric period 2 automorphism, namely 
the map $\theta_X$ above.   This gives an action of  the group $\Zdb_2$ as complete isometries on $X_c$, via the identity map and $\theta_X$.   We will use this notation many times in our paper. 

Henceforth we shall shorten `completely reasonable' to `reasonable' or `reasonable operator space complexification', since we will want all complexifications hereafter to be completely reasonable. 
 
 \begin{proposition} \label{chco}  Let $X$ be a real operator space with a complete isometry $\kappa : X \to Y$ into a complex operator space. 
 Then $(Y,\kappa)$ is a reasonable operator space complexification  of $X$ if and only if $Y$ possesses a 
conjugate linear completely isometric period 2 automorphism whose fixed points are $\kappa(X)$.  
 \end{proposition} 
 
\begin{proof}   The one direction is easy and mentioned above.   Conversely, 
given such an automorphism $\theta$, that $Y = \kappa(X) \oplus i \kappa(X)$ is just as in the Banach space case. 
Indeed  any $z \in Y$ may be written as 
$z = x + i y  \in \kappa(X) + i \, \kappa(X),$ where $x = \frac{z + \theta(z)}{2}$ and $y = \frac{z - \theta(z)}{2i}$ are in $\kappa(X)$. 
Since $\theta$ is conjugate linear we have  $\theta(x + iy) = x - iy$, and since $\theta$ is a complete isometry  
we have that $(Y,\kappa)$ is a reasonable  complexification.
\end{proof}

Ruan showed that a real operator space 
$X$ possesses a reasonable operator space complexification, which is unique up to complete isometry.
The `existence' of such complexification follows quickly from   Proposition \ref{chco}.
Indeed suppose that  $X \subset B(H)$ for a real Hilbert space $H$, and let $H_c$ be the canonical
Hilbert space complexification of $H$.   Let  $\kappa : B(H) \to B_{\Cdb}(H_c)$ be  $\kappa(T) = T_c$. 
It is easy to check that $\kappa$  is a faithful $*$-homomorphism, and  $\theta(T) =  \theta_H \circ T \circ \theta_H$ is a period 2 conjugate linear $*$-automorphism of 
the $C^*$-algebra $B(H_c)$ whose fixed points are $\kappa(B(H))$.     Hence  $\kappa$ and $\theta$ are completely isometric.
So $(B_{\Cdb}(H_c), \kappa)$ is a  reasonable operator space complexification
of $B(H)$ by Proposition \ref{chco}.  Hence $\kappa(X) + i \kappa(X)$ is a reasonable operator space complexification of $X$.

This complexification may be identified up to real  complete isometry with the operator subspace $V_X$ of $M_2(X)$ 
consisting of matrices of the form 
\begin{equation} \label{ofr} \begin{bmatrix}
       x    & -y \\
       y   & x
    \end{bmatrix}
    \end{equation} 
    for $x, y \in X$. Note that $H_c \cong H^{(2)}$ as real Hilbert spaces
and we have 
$$B_{\Cdb}(H_c) \subset B_{\Rdb}(H_c) \cong B(H^{(2)}) \cong M_2(B(H)).$$ These identifications are easily checked to be real complete isometries.
The canonical embedding $\kappa : B(H) \to B_{\Cdb}(H_c)$ above and the associated embedding $B(H) + i B(H) \to \kappa(B(H)) + i \kappa(B(H)) = B_{\Cdb}(H_c)$, 
when viewed as a map into $M_2(B(H))$ by the identifications above, correspond to 
the map $$x + iy  \mapsto \begin{bmatrix}
       x    & -y \\
       y   & x
    \end{bmatrix} \in V_{B(H)} \subset M_2(B(H)) \; , \qquad x, y \in B(H).$$
Thus $B_{\Cdb}(H_c)$ is real completely isometric to $V_{B(H)}$.
The operator $i I$ in $B_{\Cdb}(H_c)$
corresponds to the matrix $$u = \begin{bmatrix}
       0   & -1 \\
       1   & 0
    \end{bmatrix}.$$ 
    Then  $\theta_{B(H)}(x + iy) = x -iy$ for $x, y \in B(H)$,     corresponds to the completely isometric 
    operation $z \mapsto u^* z u = - u zu$ on $V_{B(H)}$.  
     Restricting to $X$, we see that  $\theta_{X}$ is completely isometric, and $X_c$ 
    is real completely isometric to $V_X$.
    
\begin{theorem} \label{rcth} {\rm (Ruan's unique complexification theorem) } \ Let $X$ be a real operator space. Then $X$ possesses a  reasonable operator space complexification, which is unique up to complete isometry.   That is if $Y_1, Y_2$ are complex operator spaces such that for $k = 1, 2$ if $u_k : X \to Y_k$ is a real linear complete isometry
with $u_k(X) + i u_k(X) = Y_k$, and such that $u_k(x) + i u_k(y) \mapsto u_k(x) - i u_k(y)$ is a complete isometry, 
then there exists a unique surjective complex linear  complete isometry $\rho : Y_1 \to Y_2$ with $\rho \circ u_1 = u_2$.
 \end{theorem}
 
\begin{proof}  A proof of the `easy direction', the `existence', was given above.
For the uniqueness, let $\kappa = u_k$.  It is enough to show that the map taking $\kappa(x) + i \kappa(y)$ to the matrix in (\ref{ofr}) is a complete isometry into $M_2(X)$.
Let $x = [x_{ij}], y = [ y_{ij}]$. Then $$\| [ \kappa(x_{ij})+\iota \, \kappa(y_{ij})] \| = \left\| \begin{bmatrix}
       \kappa(x_{ij})    & -\kappa(y_{ij})\\
        \kappa(y_{ij})   & \kappa(x_{ij})
    \end{bmatrix} \right\| = \left\| \begin{bmatrix}
       x    & -y \\
       y   & x
    \end{bmatrix} \right\| $$ by the simple fact  (2.2) in \cite{BNmetric2}. This concludes the proof.  
\end{proof}

  \begin{theorem} \label{crcb}
	 If $X$ and $Y$ are real operator spaces, then $(CB_{\Cdb}(X_c,Y_c ), \kappa)$ is a reasonable operator space complexification  of $CB_{\Rdb}(X,Y)$,
	 where $\kappa(T) = T_c$.
\end{theorem}

\begin{proof}  Clearly $T \mapsto T_c$ is  a complete isometry from $CB_{\Rdb}(X,Y)$ into $CB_{\Cdb}(X_c,Y_c )$.
 Conversely, there is a conjugate linear completely isometric period 2 automorphism on $CB_{\Cdb}(X_c,Y_c)$ defined by $T \mapsto \theta_Y \circ T \circ \theta_X$.
Its fixed points are precisely the $S_c$ for $S \in CB_{\Rdb}(X,Y)$.  Indeed suppose that 
$T  \circ \theta_X = \theta_Y \circ T$.   Then $T$ takes the fixed points of $\theta_X$ into the fixed points of $\theta_Y$.
Since these fixed point spaces are (the copies of) $X$ and $Y$ respectively, this  gives a map $S : X \to Y$.
Then $T$ and $S_c$ agree on $X$ and hence on $X_c = X + i X$.  
Thus $CB_{\Cdb}(X_c,Y_c) = CB_{\Rdb}(X,Y)_c$, a  reasonable complexification. \end{proof} 

\begin{corollary}  \label{recc}  If $X$ and $Y$ are real operator spaces, then $T \in CB_{\Cdb}(X_c,Y_c)$ equals $S_c$ for $S \in CB_{\Rdb}(X,Y)$ if and only if $T(X) \subset Y$, and 
if and only if $T = \theta_Y \circ T \circ \theta_X$.  
If $X = Y$ the last condition becomes $T \circ \theta_X = \theta_X  \circ  T$, which is equivalent to $T$ being $\Zdb_2$-equivariant 
with the $\Zdb_2$-action mentioned  above  mentioned above Proposition {\rm \ref{chco}}.  \end{corollary}

\begin{proof}   If $T(X) \subset Y$
then $T = (T_{|X})_c$ since both these complex linear  maps agree on $X$.   The rest is contained in the last proof. \end{proof} 

\begin{corollary}  \label{ducr}   {\rm (Ruan)}\ If $X$ is a real operator space then the operator space  dual 
$(X_c)^*$ is a reasonable operator space complexification of $X^*$, with embedding $\kappa(\varphi) = \varphi_c$.   
\end{corollary}

The complexification of a real unital operator space or operator system may obviously be taken to be a complex unital operator space or operator system. 

\section{The characterization of commutative real $W^*$-algebras} \label{chcvn}

The characterization  in this section will be used in \cite{BReal} to prove several deeper facts about real operator spaces and algebras. 
We will improve the characterization of commutative real $W^*$-algebras  from \cite{Li} (see e.g.\ 
Section 6.3 
there, and e.g. \cite{IP}).     Indeed the commutative real $W^*$-algebras  come from the 
commutative complex $W^*$-algebras regarded as real algebras.   Ideas in the proof may be interpreted as matching certain  facts in   `measurable dynamics'  as we will remark after the theorem.

  In the following we allow $(0)$ to be considered as a $W^*$-algebra.
 
 \begin{theorem}  \label{coisre}   Any  commutative real $W^*$-algebra $M$ is (real) $*$-isomorphic to the direct sum of a 
 commutative complex $W^*$-algebra, and the set of selfadjoint elements in a commutative complex $W^*$-algebra. 
 That is $M \cong L^\infty(X,\mu,\Cdb) \oplus L^\infty(Y,\nu,\Rdb)$. 
 Here $(X,\mu)$ and $(Y,\mu)$ are localizable (or semifinite) measure spaces (and we allow one of these 
 to be the empty set, and in this case interpret the $L^\infty$-space to be $(0)$). 
  \end{theorem} 

\begin{proof}   
 Let $M$ be a commutative real $W^*$-algebra.   Then $M_c$ is a commutative $W^*$-algebra 
possessing a period 2 conjugate linear $*$-isomorphism $\pi$ with $M = \{ x \in M_c : \pi(x) = x \}$.   Consider 
the projections $p \in M_c$ such that $\pi(q) = q$ for all projections $q \leq p$.   
We claim there is a largest such projection.   If $(p_i)$ is an increasing chain of such projections, let $p = \sup \; p_i$. 
Then $\pi(p) = p$.  
For a projection $q \in M_c$ such that $q \leq p$ we have $p_i q  \leq p_i$ for each $i$.    For any spectral projection $r$ of 
$p_i q$ we have $r \leq p_i$ and $\pi(r) = r$.  So $\pi(p_i q) = p_i q = p_i \pi(q)$. 
In the limit $\pi(q)  = q$.    Thus chains have upper bounds, and so by Zorn's lemma  there is a maximal 
such projection $p$.   If $r$ is another such projection, then we claim that $e = p \vee r = p + r -pr$ is also another such projection. 
Indeed for a projection $q \leq p \vee r$ we have $\pi(q p) = qp$ since $q \wedge p \leq p$. 
Similarly $\pi(qr) = qr$  and $\pi(qpr) = qpr$, and so $\pi(q) = q$.  By maximality $p = e$ and $r \leq p$.

Let $N_0 = p^\perp \, M_c$.  Assume that $p \neq 1$.  Note that $\pi(p^\perp) = p^\perp$ and $\pi(N_0) = N_0$. 
  If $r$ is any  nonzero projection in $N_0$ there must exist a projection $q \leq r$ 
with $\pi(q) \neq q$.   We claim that this implies that there exists some nonzero projection $q$ in $N_0$ with 
$\pi(q) = 1 - q \neq 0$.    Here $1$ denotes the identity of $N_0$, namely $p^\perp$. 
We consider the class ${\mathcal C}$ of nonzero  
projections $r$ in $N_0$ such that $\pi(r) \perp r$.  Given any projection $q$ in $N_0$ with 
$q \neq \pi(q)$ we claim that $q$ has a nonzero subprojection $r$  in ${\mathcal C}$.  
Indeed set $r = q - q \pi(q)$.     Note $r \neq 0$ or else $q \leq \pi(q)$ and (applying $\pi$) $\pi(q) = q$. 
This contradicts that $q \neq \pi(q)$.    Hence $\pi(r) \neq 0$.  Also 
$$r \pi(r) = q \pi(q) -  q \pi(q) -  q \pi(q)  + q \pi(q) = 0.$$ 
Thus $r \in {\mathcal C}$.  

Let $q$ be a  maximal  projection in ${\mathcal C}$.   That this exists follows by Zorn's lemma again: 
If $(p_i)$ is an increasing chain of  projections in ${\mathcal C}$, let $r = \sup \; p_i$. 
Since  $p_i \, \pi(p_i)  = 0$ we get $r \pi(r) = 0$.   
We now show that $\pi(q) = 1 - q$.   If this were false then since $q \pi(q) = 0$ we have that 
$r = 1 - q - \pi(q)$  is a nonzero projection in $N_0$.   So there exists a projection $e \leq r$ 
with $\pi(e) \neq e$.   Let $q' = e - e \pi(e)$.   Note that $q' \leq e \leq r \leq 1-q$. 
By the above $0 \neq q' \in  {\mathcal C}$. 
Then $q + q'$ is a projection dominating $q$, and $\pi(q+q') (q+q') = \pi(q) q' + \pi(q') q$. 
Now $\pi(q) q' = 0$ since $q' \leq r \leq 1 - \pi(q)$.  Applying $\pi$, we see 
that $\pi(q') q = 0$.   So $\pi(q+q') (q+q') =  0$, so that $q + q' \in  {\mathcal C}$. 
This contradicts the maximality of $q$.  

Next note that  $x \mapsto x + \pi(x)$ is a real $*$-isomorphism 
from $q M_c$ onto $\{ a \in N_0 : \pi(a) = a \}$.     Indeed this is the map 
$q x \mapsto qx \oplus (1-q) \pi (qx)$, which is a homomorphism since $q$ is central. 
And if  $a \in N_0$ with $\pi(a) = a$ 
then $qa \in q M_c = qN_0$ with $qa + \pi(qa) = qa + (1-q) \pi(a) = a$.

Claim: $\pi(x) = x^*$ on $p M_c$, so that for any $x \in p M_c$ we have  $\pi(x) = x$ if and only if $x = x^*$.   
Thus $\{ x \in pM_c : \pi(x) = x \} = (pM_c)_{\rm sa}$. 
  Note that $x \mapsto \pi(x)^*$ is a weak* continuous complex linear isometry on the  complex $W^*$-algebra $p M_c$ 
which is the identity on projections in $p M_c$. 
Hence  $\pi(x) = x^*$  on $p M_c$, by the spectral theorem and the fact that this is true for any projection in $p M_c$. 

Finally, we claim that $M$ is real $*$-isomorphic to the real $W^*$-algebra 
$(pM_c)_{\rm sa} \oplus q M_c$.  If $x \in M_c$ with $\pi(x) = x$, then $x =  xp + x p^\perp$, where $xp = \pi(x p)$ and $x p^\perp = \pi(x p^\perp) \in N_0$. 
So $$M = 
 \{ x \in pM_c : \pi(x) = x \}   \oplus \{ x \in N_0 : \pi(x) = x \}  \cong  (pM_c)_{\rm sa} \oplus q M_c.$$ 
 Since $pM_c$ and $q M_c$ are commutative 
 complex von Neumann algebras, we are done  (for the connection with localizable/semifinite measure spaces see e.g.\ \cite{BGL}).  
\end{proof} 

{\bf Remarks.}  1)\ The separable case of the last result is in \cite{Li} (see e.g.\ Theorem 6.3.6 there) and the work of Ayupov and his coauthors
(see e.g.\ \cite{ARU} and references therein), but we could not find a full proof of our statement above of the general case in those sources.
    In the noncommutative case the result fails: 
 for a start $M_{\rm sa}$ need not even be an algebra.
 For an explicit low dimensional example, the quaternions $\Hdb$ is not of form
 $M_{\rm sa} \oplus N$ for complex  von Neumann algebras $M$ and $N$.     Indeed the center of $\Hdb$ is $\Rdb 1$, so that there is no such nontrivial splitting,
 and $\Hdb$ is not a  complex  von Neumann algebra nor the selfadjoint part of one. 

\smallskip

2)\ 
Notice that in the proof of the claim of existence of the projection $q$ with $\pi(q) = 1-q$, we only used the fact that $\pi$ was a real $*$-automorphism of the complex von Neumann algebra $M_c$. 

Suppose that  $M_c = L^\infty(X, \mu)$ for a measure space $(X, \mu)$, and $\tau: L^\infty(X, \mu)\to L^\infty(X, \mu)$ is a complex $*$-automorphism with $\tau^2 = I$.
If $(X, \mu)$ is a standard measure space for example then by von Neumann's lifting theorem there is a 2-periodic measurable automorphism $\varphi$ of $(X, \mu)$ (this means the pushforward of $\mu$ under $\varphi$ is in the same class as $\mu$) that yields $\tau$ in the sense that $\tau(f) = f\circ \varphi$ for every $f\in L^\infty(X, \mu)$. 
In this case the projections $p$ and $q$ from the above proof have the following interpretations. The projection $p$ is the characteristic function of the set $X_0:=\{x\in X : \varphi(x) = x\}$ of fixed points of $\varphi$. The complement $X-X_0$ is then decomposed into a disjoint sum $X_1 \sqcup \varphi(X_1)$, where $X_1$ is a measurable subspace of $X$, whose characteristic function is the projection $q$ in the proof above.

In general, it is a simple fact that an action of $\Zdb$ or $\Zdb_n$, when $n$ is prime, is (essentially) free if and only if it has trivial kernel (i.e. is faithful).
Thus, any such action on a measure space $(X, \mu)$ yields a decomposition $X= E \sqcup F$ where the restriction of the action on $E$ is trivial and on $F$ is (essentially) free. Arguing similarly to the one in the proof above, one can show that for any action $\Zdb_n \acts (X, \mu)$, with $n\in\Ndb$ prime, given by the measure automorphism $\varphi: X\to X$,
there is a ``periodic decomposition'' of the form $$X= X_0 \sqcup X_1 \sqcup \varphi(X_1) \sqcup \varphi^2(X_1) \sqcup \cdots \varphi^{n-1}(X_1),$$ where the restriction of $\varphi$ to $X_0$ is the identity map.
These dynamical facts are surely well-known (although through a relatively quick search we were not able to locate the precise statements in standard literature).
One should also  note that the automorphism $\pi$ in the proof of the Theorem \ref{coisre} is not $\Cdb$-linear, so that  it does not come directly 
from a map on the underlying sets; and moreover there is no obvious appropriate map on the underlying sets without restrictions  on the measure space such as being standard.
Probably one may obtain such a map using 6.3.3 in \cite{Li}, and this would give an alternative `dynamics route' to the theorem.

\section{Complexifications, envelopes, and $G$-spaces} \label{Coen} 
We begin the section gently with some of the basic ideas in the operator space case, before throwing in the $G$-space technology which for non-experts may take some time to absorb.

 Ruan shows \cite{RComp, Sharma} that a real operator space $X$ is injective if and only if $X_c$ is complex injective.  Indeed an injective 
real operator space is real complemented in $B(H)$ so that $X_c$ is complemented in $B(H_c)$.  The other direction can be deduced from the next result and the fact that
$X$ is real complemented in $X_c$. 

\begin{lemma} \label{lem1}  A complex operator space is real injective if and only if it is complex injective.    \end{lemma} 

\begin{proof}   
A   complex injective operator space $X$ is (complex, hence real) completely isometrically embedded and complemented in $B_{\Cdb}(H)$ for a complex Hilbert space $H$.  Define a  projection $P : B_{\Rdb}(H) \to B_{\Cdb}(H)$ by $P(T)(x) = \frac{1}{2} (T(x) - i T(ix))$.   It is an exercise to check that 
this is well defined and completely contractive.    It follows that $X$ is real completely isometrically embedded and complemented in the real injective space 
$B_{\Rdb}(H)$, and so is real injective.    Note that $B_{\Rdb}(H)$ is real injective since we may view $H$ as the real Hilbert space $\ell^2_{\Cdb}(I) \cong \ell^2_{\Rdb}(I \times \{ 1, 2 \})$. 

 A  real injective complex  operator subspace $X \subset B_{\Cdb}(H)$ for a complex Hilbert space $H$,
 is real linear completely contractively complemented in $B_{\Cdb}(H)$ via a  projection $P : B_{\Cdb}(H) \to X$.
 Define  $Q(x) = \frac{1}{2} (P(x) - i P(ix))$, which is  completely contractive and maps into $X$.   
 We have $Q(ix) = i Q(x)$, so that $Q$ is complex linear.
 Note that $Qx  = \frac{1}{2} (x - i(ix)) = x$ for $x \in X$, so $Q$ is a projection.
  It follows that $X$ is complex completely isometrically embedded and complemented in the complex injective space 
$B_{\Cdb}(H)$, and so is complex injective.    \end{proof}

There is no need to state real operator system versions of the above results, since an operator system ${\mathcal S}$ is injective in the category of operator systems and (necessarily selfadjoint--see \cite[Lemma 2.3]{BT}) UCP maps 
if and only if it is injective in the category of operator spaces and completely contractive maps.   Indeed if ${\mathcal S}$ is an operator system in $B(H)$ and is injective as an operator system 
(resp.\ space) 
then its complementability in $B(H)$ ensures that it is injective as an operator space (resp.\ system).  We are using the fact mentioned 
in the introduction that  a  map on a unital real or complex $C^*$-algebra is UCP  if and only if it is 
completely contractive and unital.   Indeed injective operator systems  have a $C^*$-algebra structure
(see e.g.\ \cite[Proposition 4.4]{ROnr}, and \cite[Theorem 4.9]{Sharma} in the real case), and so can be represented as a $C^*$-subalgebra of $B(H)$.

 For a real operator space $X$ the next  result, like many results involving the injective envelope, 
 may also be proved by first proving the result in the case that $X$ is an operator system,
 and second,  applying the first case to the 
 Paulsen operator system ${\mathcal S}(X)$.  See  for example \cite{BP}, or in  4.4.2 of \cite{BLM} or Chapter 16 of \cite{Pau}, and its applications in those 
sources.     We also remark that some parts of this result and the next are generalized later for $G$-spaces, however we do the `base case first' because of its 
importance and clarity,  
and also to give the model that makes the later proofs quicker to grasp.

\begin{theorem} \label{ijco}   Let ${\mathcal S}$ be a  real operator space with injective envelope $(I({\mathcal S}), j)$.
\begin{itemize}
\item [(1)]  $I({\mathcal S})_c = I({\mathcal S}_c)$ (that is, $(I_{\Rdb}({\mathcal S})_c, j_c)$ is a complex  injective envelope of ${\mathcal S}_c$).
\item [(2)]   If ${\mathcal S}$ is
a real operator system or unital real  operator space  then  $I({\mathcal S})$ may be taken to be a 
real $C^*$-algebra, indeed a real $C^*$-subalgebra of its  $C^*$-algebra complexification, 
which is an injective envelope of ${\mathcal S}_c$.  Also $j$ may be taken to be  unital.
\item [(3)] If ${\mathcal S}$ is an approximately unital real  operator algebra
(or an approximately unital real Jordan operator algebra) then the conclusions in the first sentence of  {\rm (2)} hold, but  in addition 
   $j$ may be taken to be a homomorphism. 
\end{itemize} 
\end{theorem}

\begin{proof}  (1)\ Let  $Z = I({\mathcal S})$.  
To see that 
$(I({\mathcal S})_c, j_c)$ is an injective envelope of ${\mathcal S}_c$, by the lines above Lemma \ref{lem1} we have that  $I({\mathcal S})_c$ is an 
injective extension  of ${\mathcal S}_c$.  
To see that it is rigid, suppose that $u : Z_c \to Z_c$ is a (complex linear) complete contraction  
extending $I_{{\mathcal S}_c}$.    Then $R = \frac{1}{2} (u + (\theta_Z \circ u  \circ \theta_Z))$  is a $\Zdb_2$-equivariant (as in Corollary \ref{recc})  complex linear complete contraction $v$ 
 extending $I_{{\mathcal S}_c}$. 
 Thus $R = v_c$ by the just cited Corollary \ref{recc}, for a real linear complete contraction on $I({\mathcal S})$ extending $I_{{\mathcal S}}$. 
  By rigidity we have $v = I$ and so $R = I$.  Now the identity is well known to be an extreme point of the unit ball of any unital (real or complex) Banach algebra 
(see e.g.\ \cite[Corollary 2.1.42]{Rod}, although there is a well known simple and direct Krein-Milman argument). 
Hence $u = I$.  
Thus $I({\mathcal S})_c$ is an injective and rigid extension  of ${\mathcal S}_c$, and hence it is
an injective  envelope of ${\mathcal S}_c$.

(2)\ If ${\mathcal S}$ is 
a real operator system or unital real  operator space, then as in the complex case $I({\mathcal S})$ may be taken to be a real unital $C^*$-algebra $A$ with $j(1) = 1$ (see 
\cite[Theorem 4.9]{Sharma} and the proof of \cite[Corollary 4.2.8 (1)]{BLM}).  The $C^*$-algebra  $(A_c,j_c)$ is a complexification of $A$ with the same identity, 
and $I({\mathcal S})$ is  a real $C^*$-subalgebra.   If ${\mathcal S}$ is 
a real unital operator algebra then as in the complex case \cite[Corollary 4.2.8 (1)]{BLM} we may further take $j$ to be a homomorphism. 

(3)\  Let $A$ be an approximately unital real (Jordan) operator algebra, represented completely isometrically as a  Jordan subalgebra of $B(H)$.     
Its unitization $A^1 = A + \Rdb I_H$ is (as an operator space and (Jordan)  algebra) independent of $H$ by \cite[Lemma 4.10]{BT}.
Using \cite[Proposition 4.3]{BT} 
and  taking a compression, we may assume that a partial cai $(e_t)$ of $A$ has SOT limit $1$. 
Then $I(A)$ may be identified with the range of a  completely contractive projection  $\Phi$  on $B(H)$ fixing $A$.  Let $\varphi = 
\langle \Phi(\cdot) \zeta, \zeta \rangle$, where $\| \zeta \| = 1$.  As in 
the proof of \cite[Corollary 4.2.8]{BLM} this is a contractive functional on $A + \Rdb I$, and the fact that $\varphi(e_t) \to 1$ implies that 
$\| \varphi \| = \| \varphi_{|A} \|$.   This implies by  \cite[Lemma 4.11]{BT} that $\varphi(I) = \langle \Phi(I) \zeta, \zeta \rangle = 1.$  
By the converse to the Cauchy-Schwarz inequality, $\Phi(I) \zeta = \zeta$.  So $\Phi(I) = I$. 
This implies as in \cite[Corollary 4.2.8]{BLM}  that $(I(A^1),j)$ is a $C^*$-algebra and injective envelope of $A$,
 and $j$ is a homomorphism.   By the unital case above, $I(A)$ is a unital real $C^*$-subalgebra of 
$I(A^1_c)$ and we have $I(A)_c = I(A^1)_c = I(A^1_c) = I(A_c)$ as real $C^*$-algebras.
 \end{proof}

\begin{corollary} \label{tocenv} Let $A$ be  a real operator system, unital operator space, or approximately unital real  operator algebra
(or Jordan operator algebra), with $I(A)$ taken to be a $C^*$-algebra as in  Theorem {\rm \ref{ijco}}.  
\begin{itemize}
\item [(1)]  If $C^*_e(A)$ is the $C^*$-subalgebra of $I(A)$ generated by $A$
then $C^*_e(A)_c = C^*_e(A_c)$.    
\item [(2)]   $C^*_e(A)$ has the universal  property expected 
of the $C^*$-envelope:  given any unital complete isometry 
 $j :  A \to D$ into a real $C^*$-algebra $D$ such that $j(A)$ generates $D$ as a real $C^*$-algebra,
 there exists a (necessarily unique and surjective) $*$-epimorphism $\pi : D \to C^*_e(A)$  such that $\pi \circ j$ is the canonical inclusion of $A$ in 
 $C^*_e(A)$.   
 \end{itemize} 
 \end{corollary} 

\begin{proof}    
  (1)\ The real operator system and unital operator space result here follows easily from Theorem \ref{ijco} (1) and (2).   
 It also follows from the theory of involutions similarly to
 e.g.\ \cite[Theorem 2.4 or Theorem 2.8]{BWinv} or \cite[Proposition 1.16]{BKM}.  
 It will also be generalized in Corollary \ref{tocenvG}. 
In the approximately unital case we know $(A^1)_c = (A_c)^1$ by the uniqueness of unitization (e.g.\ \cite[Lemma 4.10]{BT}).
By the proof  in Theorem \ref{ijco}, $I(A) \subset I(A_c)$ are unital $C^*$-algebras and agree
with the  injective envelope of $A^1$ and $(A_c)^1$.  In  the complex case (4.3.4 in \cite{BLM}), 
  $C^*_e(A_c)$ is defined to be the $C^*$-algebra generated by $A_c$ inside 
$C^*_e(A^1_c)$.   Since $$C^*_e(A^1_c) = C^*_e(A^1)_c  = C^*_e(A^1) + i \, C^*_e(A^1) \subset I(A^1)_c = I(A)_c,$$ we have $C^*_e(A_c) = C^*_e(A) + i \, C^*_e(A)
 = C^*_e(A)_c$.

(2)\  This is already mentioned in \cite[Section 4]{Sharma}, and also follows immediately from the complex case by complexification.  
 \end{proof} 

{\bf Remark.}  For a nonunital real operator algebra $A$ with no kind of approximate identity one defines 
$C^*_e(A)$ to be the $C^*$-algebra generated by $A$ inside 
$C^*_e(A^1)$.  Then as in \cite[Proposition 4.3.5]{BLM} it is easy to check that 
 $C^*_e(A)$ has the universal  property of the $C^*$-envelope:  given any unital completely isometric homomorphism  
 $j :  A \to D$ into a real $C^*$-algebra $D$ such that $j(A)$ generates $D$ as a real $C^*$-algebra,
 there exists a (necessarily unique and surjective) $*$-epimorphism $\pi : D \to C^*_e(A)$  such that $\pi \circ j$ is the canonical inclusion of $A$ in 
 $C^*_e(A)$.  Again we use the uniqueness of unitization (e.g.\  \cite[Lemma 4.10]{BT}).

    \bigskip

We now turn to the more general setting of $G$-operator spaces.   The reader who only cares about real spaces for example may take $G = (0)$.

For a given group $G$, a $G$\textit{-set} refers to a given action on a set $X$. Given two $G$-sets $X\text{ and }Y$, a \textit{$G$-equivariant map} (or \textit{$G$-map}) is a map $f:X\to Y$ such that $$f(gx)=gf(x),\quad\text{ for all }g\in G\text{ and }x\in X.$$
Henceforth let $G$ be a discrete group (although as we said in the Introduction most of the results in this section hold in far greater generality). 
 As in  Hamana's notion of $G$-module from \cite{Hamiecds,Hamiods} we define a real $G$-{\em operator space} (resp.\ $G$-{\em operator system}) to be a real operator space (resp.\  operator system) $X$ such that 
 $G \acts X$ by real linear surjective complete isometries (resp.\   unital complete order real linear isomorphisms).   Sometimes we 
 write these complete isometries as $u_g : X \to X$ for $g \in G$.   
 We say that a complex operator space $X$ is a {\em complex} $G$-{\em operator space} if $G\acts X$ by complex  linear surjective complete isometries.
A    $G$-{\em unital operator space} is a unital operator space and a $G$-operator space with  a $G$-action by  linear surjective unital complete isometries.
 We remark that in \cite{Hamiods} Hamana studies a class of complex operator spaces much more general than the complex $G$-operator spaces and works out the injective and ternary 
 envelope theory for these.  
 
 We recall that a (real or complex) $G$-$C^*$-algebra is a (real or complex)  $C^*$-algebra $B$ such that 
 $G \acts B$ by (real or complex)  $*$-automorphisms.   Clearly a $G$-$C^*$-algebra (resp.\ unital $G$-$C^*$-algebra) is 
 a $G$-operator space (resp.\ $G$-operator system).  Conversely a  $G$-operator system which is a $C^*$-algebra is a $G$-$C^*$-algebra.
 Indeed the `operator system 
$G$-action' on a $C^*$-algebra is necessarily by $*$-automorphisms.
This is because of the operator space version of the Kadison-Banach-Stone theorem that we have seen several times before in the complex case:
a unital complete isometry between $C^*$-algebras is a $*$-homomorphism.  
The real case of this version  of the Kadison-Banach-Stone theorem follows immediately 
from the complex version by complexification (see also \cite[Theorem 4.4]{RComp}). 
Similarly, we define a $G$-TRO to be  a (real or complex)  TRO  $Z$ such that 
 $G \acts Z$ by (real or complex)  ternary automorphisms (see also \cite[Section 4]{Hamiods}).   By the facts about TRO's mentioned in the introduction, 
  a $G$-TRO  is 
 a $G$-operator space.  Conversely a  $G$-operator space which is a TRO is a $G$-TRO.
 Indeed the `operator space $G$-action' on a TRO is necessarily by ternary automorphisms, by the just cited facts.

\begin{lemma} \label{Gcom} If $X$ is a  real $G$-operator space (resp.\ real $G$-operator system) then $X_c$ with action $g(x+iy) = gx + i gy$ is a complex $G$-operator space (resp.\  complex $G$-operator system), $\theta_X$ is $G$-equivariant, and the canonical projection
$X_c \to X$ is a real linear $G$-equivariant  completely contraction (resp.\ UCP).  Indeed the $G$-action on $X_c$ commutes with the $\Zdb_2$-action mentioned above Proposition {\rm \ref{chco}}.
\end{lemma}

\begin{proof}   This is essentially just the fact that $(u_g)_c : X_c \to X_c$ is a surjective complete isometry.  
In the system case the canonical projection
$X_c \to X$ is  completely contractive and unital, hence UCP.
 \end{proof} 

{\bf Remarks.} 1)\   If $X, Y$ are  real $G$-operator spaces (resp.\ real $G$-operator systems) and $T : X \to Y$ is
$G$-equivariant then so is $T_c$.

\smallskip

2)\   A simple but important fact is that 
if $X$ and $Y$ are real $G$-operator spaces, then a $G$-equivariant 
$T \in CB_{\Cdb}(X_c,Y_c)$ equals $S_c$ for  a $G$-equivariant $S \in CB_{\Rdb}(X,Y)$ if and only if $T$ is 
$\Zdb_2$-equivariant (as defined above Proposition {\rm \ref{chco}}).   

One can restate many of our definitions and results in this paper in terms of the $\Zdb_2$-action discussed above Proposition {\rm \ref{chco}}.  For example,  the maps $P$ and $Q$ in Lemma \ref{lem1} (and elsewhere in our paper)
 may be viewed as an average with respect to the $\Zdb_2$-action.
Or, Corollary  \ref{recc}  may be stated as saying that $T \in CB_{\Cdb}(X_c,Y_c)$ equals $S_c$ for $S \in CB_{\Rdb}(X,Y)$ if and only if  $T$ is $\Zdb_2$-equivariant.

\smallskip

3)\ 
One can often make a $G$-map out of an arbitrary linear map by ``averaging".
For example suppose that $G$ is a finite group, and $X$ and $Y$ are vector spaces, which are also $G$-sets, and $T:X\to Y$ is linear. Define  $\tilde{T}:X\to Y$ by $$\tilde{T}(x)=\frac{1}{| G |}\sum_{g\in G}g^{-1}T(gx),$$ for  $x\in X$. Then $\tilde{T}$ is  linear  and $G$-equivariant.

\bigskip

 Hamana proves in  \cite{Hamiecds}  and \cite{Hamiods}  that injective envelopes and ternary envelopes
 exist in his categories of  (complex) $G$-modules and $G$-morphisms.
  This theory parallels the usual injective envelope theory, with 
 $G$-{\em rigid} and $G$-{\em essential} playing their requisite roles just as before.  (Note for example that $G$-rigid is defined just as in the definition of rigid, but with all morphisms
$G$-equivariant.) The same theory also exists with the same proofs for real  
 $G$-operator spaces (resp.\ real $G$-operator systems, complex $G$-operator spaces), where the morphisms for real or complex $G$-operator spaces
 (resp.\ $G$-operator systems) are $G$-equivariant real or complex linear complete contractions (resp.\   UCP maps).  We call these the
 $G$-{\em morphisms}.  See \cite{CeccoTh} for a systematic development of this theory in very many categories.
  Hamana's  complex $G$-modules in  \cite{Hamiecds}  are the complex $G$-operator systems in this notation. 
 See \cite{KK} and its sequels for more modern usage of the $G$-injective envelope. 
 
 \begin{theorem} \label{ath} {\rm \cite{Hamiecds,Hamiods, CeccoTh}} \ 
	Every real  or complex $G$-operator space (resp.\ $G$-operator system) $X$ has a real  or complex $G$-injective envelope, written $(I_{G}(X),\kappa)$ (in the category of real  or complex  
 $G$-operator spaces (resp.\ $G$-operator systems).  Here 
 $\kappa : X \to  I_{G}(X)$ is a real or complex linear $G$-equivariant complete isometry (resp.\ unital complete order embedding).
 This $G$-injective envelope is unique in the sense that for any  $G$-injective envelope $(Z, \lambda)$ of $V$ in  the
  categories above, there is a $G$-isomorphism $\psi: I_{G}(V) \rightarrow Z$ 
 (that is, $\psi$ and $\psi^{-1}$ are morphisms in the category) with $\psi \circ \kappa=\lambda$.   
\end{theorem}

We summarize briefly the main ideas, which will also be used below.  For any $X$ in one of the just mentioned categories
 (e.g.\ real $G$-operator spaces), if $X \subset  B(H)$ as an operator subspace (or subsystem), we may define a $G$-action  on $l^\infty(G,X)$ and
 $l^\infty(G,B(H))$ by $(t f)(s) = f(t^{-1} s)$ for $s, t \in G$.  There is also an involution on $l^\infty(G,B(H))$: $f^*(s) = f(s)^*$. 
 We have $G$-equivariant inclusions $$X  \; \; \subset   \; \;  l^\infty(G,X)  \; \; \subset   \; \;   l^\infty(G,B(H)) ,$$ 
 where the first of these is the map $j(x)(s) = s^{-1} x$  for $s \in G, x \in X$.   We sometimes write this $j$ as $j_X$ below.
 It is easy to see that $l^\infty(G,B(H))$ is an injective and  $G$-injective von Neumann algebra, 
 and $l^\infty(G,X)$ is $G$-injective if $X$ is injective (see \cite[Lemma 2.2]{Hamiecds} for the idea of proof).
 The $G$-injective envelope may be constructed as a subspace of $W = l^\infty(G,B(H))$, namely as the range of
 a certain completely contractive projection $\Phi : W \to W$ that fixes $j(X)$.  Using these  ideas (see  \cite[Lemma 2.4]{Hamiecds} 
 and the lines above it) one adapts the usual construction of the injective envelope \cite{BLM,ER,Pau} to
$G$-spaces to obtain Theorem \ref{ath}.   See also e.g.\ \cite{Bry} for an exposition of Hamana's ideas in one particular case. 

Note that if $X$ is a (real or complex) unital $C^*$-subalgebra (resp.\ von Neumann subalgebra,
unital subalgebra, unital subspace, operator subsystem, subTRO)  of $B(H)$ then $l^\infty(G,X)$ is a  unital $C^*$-subalgebra (resp.\ von Neumann subalgebra, unital
subalgebra, unital subspace, operator subsystem, subTRO)  of $l^\infty(G,B(H))$.  It  is easy to see that $j_X : X \to l^\infty(G,X)$ is a unital $*$-homomorphism
(resp.\ normal $*$-homomorphism, unital homomorphism, unital map, UCP map, ternary morphism).  (The von Neumann algebra case of this will not be used in this paper, but
can be deduced from a fact from 1.6.6 in \cite{BLM}.)
 
 As in  \cite[Remark 2.3]{Hamiecds} we have: 
 
 \begin{lemma} \label{Hrem}  Let $X$ be a  real  or complex $G$-operator space (resp.\ real or complex $G$-operator system).
 Then $X$ is   $G$-injective in its given category if and only if $X$ is injective as a real  or complex operator space (resp.\ real or complex operator system) 
 and there is a $G$-morphism $\phi : l^\infty(G,X) \to X$ such that $\phi \circ j = I_X$.
 \end{lemma} 

It follows that a $G$-injective real  or complex $G$-operator space $Z$ is a TRO, 
since  injective  operator spaces are TRO's (we recall from  e.g.\ 4.4.2--4.4.3   in   \cite{BLM} and \cite{Sharma} (the paragraph before Theorem 4.13 there) that  
the $1$-$2$-corner $p I({\mathcal S}(X)) p^\perp$ may be taken to be a TRO and an injective envelope of $X$). 
Hence $Z$ is a $G$-TRO by the lines above Lemma \ref{Gcom}. 

 As mentioned, $G$-{\em rigidity} and $G$-{\em essentiality} play their requisite roles just as before: An injective $G$-extension is a $G$-injective envelope if and only if it is a $G$-rigid extension, and  if and only if it is $G$-essential extension.   This follows a well-trodden route,  see also \cite{CeccoTh} for precise details if needed in certain particular categories. 
For a $G$-operator system ${\mathcal S}$ one may take the $G$-injective envelope to be a unital $G$-$C^*$-algebra (this is done as usual with the Choi-Effros product, as
in  the lines above \cite[Theorem 2.5]{Hamiecds}).  

\begin{lemma} \label{Hreml}  Let $X$ be a  real  $G$-operator space (resp.\  $G$-operator system, unital $G$-operator space).
 Then $l^\infty(G,X)_c =  l^\infty(G,X_c)$ via a completely isometric complex $G$-morphism  $\kappa$ (that is, 
 a $G$-isomorphism) with
 $\kappa \circ (j_X)_c = j_{X_c}$.
 \end{lemma} 
 
\begin{proof}   The automorphism $\theta_X$ induces a conjugate linear period 2 completely 
isometric  $G$-equivariant automorphism on $l^\infty(G,X_c)$ whose fixed points correspond to  $l^\infty(G,X)$. 
Here $$\kappa(f + ig)(t) = f(t) + i g(t) , \qquad f, g \in l^\infty(G,X), \; t \in G .$$ For $x,y \in X$
we have $$(\kappa((j_{X})_c(x + iy)))(t) = (\kappa(j_{X}(x) + i \, j_{X}(y)))(t)  = j_{X}(x) (t) + i \, j_{X}(y)(t) ,$$
which is $t^{-1} x + i \,t^{-1} y =   j_{X_c}(x+iy)(t).$  So $\kappa \circ (j_X)_c = j_{X_c}$.
 \end{proof}

 \begin{corollary} \label{Gin} If $X$ is a  real $G$-operator space (resp.\ real $G$-operator system) then $X_c$ is complex 
$G$-injective  if and only if $X$ is real $G$-injective.  \end{corollary} 

\begin{proof}    By Lemma \ref{Hrem} we have that $X$ is real $G$-injective if and only if $X$ is injective as a real  operator space  
 and there is a real $G$-morphism $\phi : l^\infty(G,X) \to X$ such that $\phi \circ j = I_X$. 
 By  the fact above  Lemma \ref{lem1} 
 this 
 implies that $X_c$ is injective as a real  operator space.
 Also by Lemma \ref{Hreml}  from 
   $\phi_c : l^\infty(G,X)_c  \to X_c$ we obtain   a complex $G$-morphism 
  $\psi = \phi_c \circ \kappa^{-1} : l^\infty(G,X_c)  \to X_c$ with   $\psi \circ j_{X_c}  = I_{X_c}$. 
 So $X_c$ is complex $G$-injective  by Lemma \ref{Hrem} again.

 Conversely if $X_c$ is complex $G$-injective, then it is real $G$-injective by e.g.\ a simpler variant of the arguments in the last paragraph
 (briefly, it is complex injective so real injective, so $l^\infty(G,X_c)$ is real $G$-injective and hence so is $X_c$).
So $X$  is real $G$-injective  since it  is `real $G$-complemented' in $X_c$ (see Lemma \ref{Gcom}).
\end{proof}

\begin{proposition} \label{Ginjuos} If $X$ is a  real  (resp.\ complex)  $G$-unital operator space 
then the space  $X^*$ of adjoints is a real  (resp.\ complex) $G$-unital operator space, $S = X + X^*$ is a real  (resp.\ complex) $G$-operator system, and 
$$I_G(X + X^*) = I_G(X) \;    , \; \;  \; C^*_{e,G}(X + X^*) = C^*_{e,G}(X).$$ 
Moreover any completely contractive unital $G$-equivariant map $T : X \to Y$ between real  (resp.\ complex) unital $G$-operator spaces
extends uniquely to a selfadjoint $G$-equivariant UCP map $\tilde{T} : X + X^* \to Y + Y^*$, which is also completely isometric
if $T$ is. 
\end{proposition}

\begin{proof}   
For complex  unital operator spaces that $X + X^*$ is well defined as an operator system independent of representation  is well known \cite{Pau,BLM,ER} ,
as is the last assertion.   
In the real case this is the unital space assertion in Proposition 2.4 and Corollary 2.5 in \cite{BT}, and the lines after that.  The $G$-equivariance of the extension to  $X + X^*$ is
an exercise (some of the lines below may help).

As in the construction below Theorem \ref{ath}, we have $G$-equivariant unital completely isometric
inclusions $$X  \; \; \subset   \; \;  l^\infty(G,X) \; \; \subset   \; \;  l^\infty(G,B(H)) , $$ 
 where the first of these is the map $j(x)(s) = s^{-1} x$  for $s \in G, x \in X$.   
 Then $$j(X)^* \; \; \subset \; \; l^\infty(G,B(H)) \; , \; \; \; \; j(X) \, + \, j(X)^* \; \; \subset \; \; l^\infty(G,B(H)) $$ are  $G$-equivariant inclusions. 
 For example, for $s, t \in G, x \in X$ we have $$(t j(x)^*)(s) = j(x)^*(t^{-1} s) = (j(x) (t^{-1} s))^* = ((s^{-1} t) x)^* = (s^{-1} t) x^* ,$$ while 
 $$j(tx)^*(s) = (j(tx)(s))^* = (s^{-1} (tx))^*  = s^{-1} (t x^*)  = (s^{-1} t) x^* .$$ 
 Thus  $t j(x)^* = j(tx)^*$.    Hence $j(X)^*$, and so $X^*$, is a real  (resp.\ complex) $G$-unital operator space.  Similarly
 $j(X) + j(X)^*$, and so $X + X^*$, is a real  (resp.\ complex) $G$-operator system.   We are also using the fact in the first line of the proof.
 
 Choose $I_G(X + X^*)$ to be a unital $G$-$C^*$-algebra as above with unital $G$-$C^*$-subalgebra $C^*_{e}(X + X^*)$.
These are clearly $G$-injective (resp.\ $G$-$C^*$-) extensions of $X$ which are $G$-rigid since they are  $G$-rigid  as extensions of 
$X + X^*$,  and $X \subset X + X^*$.  
 \end{proof} 

Because of the last proposition many results about unital $G$-operator spaces follow from the $G$-operator
 system case.   For example  Lemma \ref{Hrem} and Corollary \ref{Gin} have appropriate variants for   unital $G$-operator spaces.

The (real or complex) ternary $G$-envelope ${\mathcal T}_G(X)$ of a (real or complex) $G$-operator space may be defined to be the (real or complex) ternary subspace 
of $I_G(X)$ generated by the copy of $X$.  
Since $X$ is a $G$-subspace of $I_G(X)$ it is clear that 
${\mathcal T}_G(X)$ is also a $G$-subspace, hence is a $G$-operator space.  
 
\begin{theorem} \label{ijcoG}    Let ${\mathcal S}$ be a  real $G$-operator space.
Then $I_G({\mathcal S})_c = I_G({\mathcal S}_c)$ and ${\mathcal T}_G({\mathcal S})_c = {\mathcal T}_G({\mathcal S}_c)$.   
 If ${\mathcal S}$ is
a real $G$-operator system or unital real  $G$-operator space then  $(I_G({\mathcal S}),j)$ may be taken to be a 
real $G$-$C^*$-algebra with $j$ unital,  indeed $I_G({\mathcal S})$ is a real $G$-$C^*$-subalgebra of  its complexification 
$I_G({\mathcal S}_c)$.  
\end{theorem}

\begin{proof}   The first part follows as in the proof of Theorem \ref{ijco}.  For example, 
by Corollary \ref{Gin}, $I_G({\mathcal S})_c$ is a $G$-injective extension  of ${\mathcal S}_c$. 
If $u : I_G({\mathcal S})_c \to I_G({\mathcal S})_c$ is a (complex linear) $G$-morphism 
extending $I_{{\mathcal S}_c}$
then as in the proof of \ref{ijco} 
we have that  $R = \frac{1}{2} (u + (\theta_{{\mathcal S}} \circ u  \circ \theta_{{\mathcal S}})) = I$ and hence $u = I$.
 Thus $I_G({\mathcal S})_c$ is a $G$-injective $G$-rigid extension  of ${\mathcal S}_c$, and hence it is
a $G$-injective  envelope of ${\mathcal S}_c$.

As stated after Lemma \ref{Hreml}, 
$Z = I({\mathcal S})$  may be taken to be a $G$-TRO.  
Since 
  ${\mathcal T}_G({\mathcal S})$ is 
the real $G$-subTRO generated by ${\mathcal S}$, we have that  $Z_c = {\mathcal T}_G({\mathcal S}) + i {\mathcal T}_G({\mathcal S})$ is a complex $G$-TRO and  
$G$-extension of ${\mathcal S}_c$, and $Z_c$ is  easily seen to be 
the $G$-TRO generated by ${\mathcal S}_c$ in 
$I_G({\mathcal S})_c$. So ${\mathcal T}_G({\mathcal S})_c = {\mathcal T}_G({\mathcal S}_c)$.

That   $(I_G({\mathcal S}),j)$ may be taken to be a 
$G$-$C^*$-algebra for $G$-unital spaces uses the same idea for \cite[Corollary 4.2.8 (1)]{BLM} or the lines above \cite[Theorem 2.5]{Hamiecds}.  
Indeed the projection $\Phi$ from  $W = l^\infty(G,B(H))$ onto $I_G({\mathcal S})$ in the construction of the $G$-injective envelope discussed 
after Theorem \ref{ath} above is unital, so UCP, and so $I_G({\mathcal S})$ is a $C^*$-algebra in  the Choi-Effros product (see the next proof for a few more details
if needed).  \end{proof} 
 
{\bf Remarks.}   1)\ Similarly there will be `$G$-versions' of the statements in   Theorem \ref{ijco} for  unital real  operator algebras or Jordan operator algebras.

\smallskip

2)\ For a finite group $G$ one may use Theorem \ref{Ginj} to give shortened proofs of  Theorem \ref{ijcoG} and Corollaries \ref{Gin}, \ref{lem1G}, and \ref{corcpxG}. 

\bigskip

 The $G$-$C^*$-envelope $C^*_{e,G}(X)$ of a  $G$-unital operator space or system $X$ 
is defined to be the (real or complex) $C^*$-algebra 
generated by the copy of $X$ in $I_G(X)$.   Nearly all of the following result is due to Hamana in the complex case (who proved it  in \cite{Hamiods} for general locally compact group actions and 
indeed much more generally than that).

\begin{corollary} \label{tocenvG} Let $X$ be  a real $G$-operator system or  $G$-unital operator space,  
with $I_G(X)$ taken to be a $C^*$-algebra as in  Theorem {\rm \ref{ijcoG}}.  \begin{itemize}
\item [(1)] 
$C^*_{e,G}(X)_c = C^*_{e,G}(X_c)$.    
\item [(2)]  $C^*_{e,G}(X)$ has the universal property:  given any unital $G$-equivariant complete isometry 
 $\kappa :  X \to D$ into a real $G$-$C^*$-algebra $D$ such that $\kappa(X)$ generates $D$ as a real $C^*$-algebra,
 there exists a (necessarily unique and surjective) $G$-equivariant  $*$-epimorphism $\pi : D \to C^*_{e,G}(X)$  such that $\pi \circ \kappa$ is the canonical inclusion of $X$ in 
 $C^*_{e,G}(X)$.    
 \item [(3)] ${\mathcal T}_G({\mathcal S})$ in Theorem {\rm  \ref{ijcoG}} has the desired universal  property of the ternary $G$-envelope. 
 Namely:  given any $G$-equivariant complete isometry 
 $\rho :  {\mathcal S} \to W$ into a real $G$-TRO $W$ such that $j({\mathcal S})$ generates $W$ as a real $G$-TRO, 
 there exists a $G$-equivariant ternary morphism $\pi : W \to {\mathcal T}({\mathcal S})$ such that $\pi \circ \rho = j$.  
\end{itemize}
 \end{corollary} 

\begin{proof}   Essentially we just need to explicate Hamana's method in the present setting.  

(1)\  Let $A$ be  a real unital $G$-operator space or system, and $I_G(A_c)$ a unital $G$-$C^*$-algebra as in the last result. 
Since the $C^*$-subalgebra of a unital $C^*$-algebra
 generated by a  subspace containing the identity 
is the TRO generated by  that subspace,  $C^*_e(A)_c = C^*_e(A_c)$ as complex $G$-$C^*$-algebras and as $G$-extensions of $A_c$
by the 
analogous statement for the $G$-ternary envelope in Theorem \ref{ijcoG}.

(2)\ For the universal property we follow the argument in
4.3.3 in \cite{BLM}.
First we choose a $G$-injective envelope $I_G(X)$ of $X$
which is a unital $G$-$C^*$-algebra containing $X$ unitally, as in Theorem \ref{ijcoG}. Then $C^*_{e,G}(X)$ is a $C^*$-subalgebra of  $I_G(X)$.
 Suppose  that $(D,\kappa)$ is any $G$-$C^*$-extension of a
(real or complex) $G$-operator
system $X$, with $D$ generated by $\kappa(X)$, 
and suppose that $D$ is a $G$-$C^*$-algebra which is a unital $*$-subalgebra of $B(H)$.  Then $j_D(\kappa(X)) \subset l^\infty(G,B(H))$.  By
the construction
of the $G$-injective envelope discussed after Theorem \ref{ath},
there is a $G$-equivariant
completely positive idempotent map $\Phi$ on
$l^\infty(G,B(H))$ whose range is a $G$-injective envelope $R$ of $j_D(\kappa(X))$.  Also, $R$
is a $C^*$-algebra with respect to a new product, the Choi-Effros product. With respect to the
usual product on $l^\infty(G,B(H))$, the $C^*$-subalgebra of $l^\infty(G,B(H))$
generated
by $R$ contains $j_D(D)$, the $C^*$-subalgebra of $l^\infty(G,B(H))$ generated by
$j_D(\kappa(X))$. We let $B$ be the $C^*$-subalgebra (in the new product) of
$R$ generated by $j_D(\kappa(X))$.  As in 4.3.3 in \cite{BLM}, a formula related to the Choi-Effros product shows that $\pi =
\Phi_{\vert D}$ is a $*$-homomorphism from $D$ to $R$, with
respect to the new product on $R$.  Since $\pi$ extends the
identity map on $j_D(\kappa(X))$, it clearly also maps into $B$. Since the canonical inclusion of $X$ in 
 $C^*_{e,G}(X)$ is $G$-equivariant, we have $\pi(\kappa(g x) )= g \pi(\kappa(x) )$ for $x \in X$.
 Thus $\pi(\kappa(g x) y)= g \pi(\kappa(x) y)$ for $x \in X$ and $y$ a product of terms from $\kappa(X)$ and $\kappa(X)^*$. So $\pi$ is $G$-equivariant.

As in the usual argument (e.g.\ 4.3.3 in \cite{BLM}),
the natural unital $G$-equivariant completely isometric
surjection $R \rightarrow I_G(X)$ is a
$*$-homomorphism. Hence it is clear that $(B, j_D \circ \kappa)$,
as a  $G$-$C^*$-extension of $X$, is identifiable with $C^*_{e,G}(X)$.
Putting these facts together, we see that $C^*_{e,G}(X)$ has the
desired universal property.

(3)\ The universal property of ${\mathcal T}_{G}(X)$
can be proved either similarly to our proof in the last paragraphs but using Youngson's theorem (\cite[Theorem 4.4.9]{BLM}, and its real variant in \cite{BReal})
in place of the Choi-Effros product and formula referred to above; 
or similarly to the proof of 8.3.11 in \cite{BLM} but with all morphisms $G$-equivariant.  The latter  is the approach taken in the 
complex case in \cite[Theorem 4.3]{Hamiods}. 
 The real case is immediate from the complex case applied to $\rho_c : {\mathcal S}_c \to W_c$ and $j_c$.
   \end{proof} 

{\bf Remarks.} 1)\ As usual, the real $G$-$C^*$-envelope of $A$ is a $G$-rigid and $G$-essential extension in the appropriate sense of the category in which we are in.
This is because it sits between  $A$ and its $G$-injective envelope, and the latter is $G$-rigid and $G$-essential.   Similarly, ${\mathcal T}_G(X)$ is a $G$-rigid and $G$-essential
extension  of the real or complex $G$-operator space $X$.

\smallskip

2)\ In particular if $(Z,j)$ is a ternary $G$-extension of ${\mathcal S}$ with the universal property 
in (3), then the usual 
commutative diagram argument  shows that  there is a $G$-equivariant ternary {\em isomorphism} $\theta : Z \to {\mathcal T}({\mathcal S})$ such that $\theta \circ j$ is the canonical inclusion of
 $X$ in 
 ${\mathcal T}({\mathcal S})$.   So $Z$ `equals' ${\mathcal T}({\mathcal S})$ as ternary $G$-extensions of $X$.

\begin{corollary} \label{lem1G}  A complex $G$-operator space or system $X$  is real $G$-injective if and only if $X$ is complex $G$-injective.   
 \end{corollary} 
 
\begin{proof}  By Lemma \ref{Hrem}, $X$ is real $G$-injective if and only if $X$ is injective as a real  operator space  or system 
 and there is a real $G$-morphism $\phi : l^\infty(G,X) \to X$ such that $\phi \circ j = I_X$.
 Now $X$ is injective as a real  operator space  or system  if and only if    $X$ is injective as a complex  operator space  or system,
 by Lemma \ref{lem1}. 
Let 
$Q(f) = \frac{1}{2} (\phi(f) - i \phi(if))$ for $f \in l^\infty(G,X)$.   As in the proof of Lemma \ref{lem1} we see that $Q$ is  complex linear  and $Q \circ j = I_X$.
It is also $G$-equivariant and completely contractive.  In the operator system case $Q(1) = Q(j(1)) = 1$, so that 
$Q$ is UCP and a $G$-morphism.  
Using Lemma \ref{Hrem} again we see that $X$ is complex $G$-injective.   The converse is similar, but easier, and as in 
the last part of Corollary \ref{Gin}. 
 \end{proof} 
 
 {\bf Remark.} A similar result holds in the unital $G$-operator space case, with a slight variant of the proof which is left to the reader.
   The same goes for the first statement in the next result.

\begin{corollary} \label{corcpxG}   The real and complex $G$-injective envelopes of a complex $G$-operator space or system  coincide. 
The real and complex ternary $G$-envelopes of a complex $G$-operator space coincide.  The real and complex  $G$-$C^*$-envelopes of a complex unital 
$G$-operator space coincide. 
\end{corollary} 

\begin{proof}  
 Let $I_G(X)$ be the complex $G$-injective envelope of the complex $G$-operator space or system
 $X$. Then $I_G(X)$ is complex $G$-injective and $G$-rigid. By Corollary \ref{lem1G}  we have that 
$I_G(X)$ is real $G$-injective. So to show that it is a real $G$-injective envelope, it is enough to show that it is real $G$-rigid. Consider $X$ and $I_G(X)$ as real spaces and consider a  real linear $G$-morphism $u : I_G(X) \to I_G(X)$ 
extending $I_{X}$.   As in the proof of 
Corollary 
 \ref{lem1G} we have that  $Q(z)= \frac{1}{2} (u(z) - i u(iz))$, for $z \in I_G(X)$, is a complex linear 
$G$-morphism extending $I_{X}$.  By 
complex rigidity and  the extreme point argument 
 in the proof of  Theorem \ref{ijco} we have $Q = I =u$.   Thus $I_G(X)$ is a real $G$-injective $G$-rigid extension  of $X$, and hence it is
a real $G$-injective  envelope of $X$.

The real  ternary $G$-envelope is the real $G$-TRO in  the real $G$-injective envelope,
hence in 
$I_{G}(X)$, generated by $X$.   Since $X$ is closed in $I_{G}(X)$ under multiplication by $i$,
this agrees with the complex $G$-TRO in  $I_{\Cdb}(X)$ generated by $X$.   So the real and complex ternary $G$-envelopes 
coincide.  The $G$-$C^*$-envelope assertion follows from this if we take $I_{G}(X)$ a $G$-$C^*$-algebra with the same identity as $X$ as in Theorem 
\ref{ijcoG}. 
\end{proof} 


\begin{theorem} \label{cist}   A complex operator space which is a real TRO is a complex TRO.
A complex operator space which is a real $C^*$-algebra  is a complex $C^*$-algebra. 
A complex operator space which is a real approximately unital operator algebra   is a complex operator algebra.
 \end{theorem} 

\begin{proof}  To see the first statement apply Corollary \ref{corcpxG} with $G = (0)$   to the ternary envelope. 
The third statement  follows from the `BRS characterization' of operator algebras \cite[Theorem 2.3.2]{BLM}.  Indeed if an approximately unital algebra  $A$ 
is a complex operator space, and in addition satisfies the `real condition' $\| x y \| \leq \| x \| \| y \|$ for $x, y \in M_n(A)$, then 
it satisfies all the hypotheses of the cited characterization of complex operator algebras. 

If $A$ is a real  $C^*$-algebra then it is a complex operator algebra $B$ by the third statement.   
We have $A = A \cap A^* = \Delta(B)$,  which is a complex $C^*$-algebra.  We remark that the 
`diagonal'  $\Delta( \cdot )$ of an operator algebra is well defined independently of representation  in both the real and complex cases
by 2.1.2 in \cite{BLM} and \cite{BReal}. 
\end{proof}

 \begin{theorem} \label{Ginj} Let $G$ be a finite group.
 If $S$ is a  real  (resp.\ complex) operator space or system which is also a  real  (resp.\ complex) $G$-operator space or system 
then  $S$ is real  (resp.\ complex) $G$-injective  if and only if $S$ is real  (resp.\ complex)  injective.
 \end{theorem}

\begin{proof}  We just do the  operator space case, the  result for $G$-operator systems is similar.
Suppose that $S$ is injective.  Let $V$ and $W$ be $G$-operator spaces, with $V$ a $G$-subspace of $W$
(that is, the inclusion map is $G$-equivariant).  Let $\phi: V \to S$ be a 
real  (resp.\ complex)  $G$-equivariant complete contraction. Then since $S$ is injective, we may extend $\phi$ to a 
real  (resp.\ complex) $\psi: W \to S$, which may not be $G$-equivariant.  For each $g\in G,$ the map $\psi_g : W \to S$ defined by $\psi_g (x) = g \psi(g^{-1}x)$ is also completely contractive, and  extends $\phi$ by $G$-equivariance of $\phi$. 
If we take the average $\rho = \frac{1}{|G|} \sum_{g\in G} \psi_g$, we get an extension of $\phi$ to a $G$-equivariant 
complete contraction  from $W$ to $S$. Thus, $S$ is $G$-injective.

Now suppose that $S$ is $G$-injective.  By the next theorem (whose proof only uses the last paragraph)  we have $I(S) = I_G(S) = S$, so that $S$ is injective.
 \end{proof}

{\bf Remark.}   A similar result holds in the unital $G$-operator space case, with basically the same proof.

 In the last result and in Theorem \ref{mkb}  we focus on finite groups.   However we remark that the results will still be valid for continuous actions of a compact group $G$, with some modifications
 in proof.   Indeed the 
average $\rho$  in the last proof may be replaced by a Bochner integral $\rho = \int_G \,  \psi_g \, dg$ with respect to normalized Haar measure, which is easily argued to exist
and be $G$-invariant.   We leave the details to the reader.

\begin{theorem} \label{mkb}  Let $G$ be a finite group and $X$ a real  or complex $G$-operator space (resp.\ $G$-operator system). \begin{itemize}
\item [(1)]   The $G$-injective envelope of $X$ in the category of real  or complex 
 $G$-operator spaces (resp.\ $G$-operator systems, $G$-unital operator spaces) 
  is an injective envelope of $X$ in that category. 
 That is,
 $I_G(X) = I(X)$.  \item [(2)]   ${\mathcal T}_G(X)$ is a real (resp.\ complex)  ternary envelope of  $X$.  
 \item [(3)]    The $G$-$C^*$-envelope $C^*_{e,G}(X)$ is a real (resp.\ complex)  $C^*$-envelope
 $C^*_{e}(X)$ for a real (resp.\ complex)   $G$-unital operator space $X$. 
 \end{itemize} 
\end{theorem}

\begin{proof}  We just prove the operator space cases, the other being similar.
For (1) let $X$ be a real  (resp.\ complex) $G$-operator space.    By rigidity, for each $g \in G$ the action of $g$ on $X$ extends uniquely to a complete 
 isometry of $I(X)$, whose inverse is the extended action of $g^{-1}$.    
 Again by rigidity,  we have $g^{-1} (h^{-1} ((hg) x)) = x$ on $I(X)$, so that $(hg) x = h(gx)$.  
 So we have a group homomorphism into the group of surjective complete isometries on $X$.  
 So $I(X)$ is a real  (resp.\ complex)  $G$-operator space.   It is $G$-rigid since it is rigid.
 Then  $I(X)$ is $G$-injective by the first paragraph of the proof of Theorem \ref{Ginj} and hence is the $G$-injective envelope.
 
  Since ${\mathcal T}_G(X)$ (resp.\ $C^*_{e,G}(X)$) is the (real or complex) ternary subspace (resp.\ $C^*$-algebra) 
of $I_G(X) = I(X)$ generated by the copy of $X$, it  is a ternary envelope (resp.\ $C^*$-envelope) of $X$.    
   \end{proof}

{\bf Remark.}   There are analogous results to the operator algebra results early in Section 4 for $G$-operator algebras.  
We  say that an operator algebra $A$ is a $G$-{\em operator algebra} if $G\acts X$ by surjective completely  isometric
homomorphisms (automorphisms).  Again if $A$ is also unital then by a well known Banach-Stone theorem  
(e.g.\ 4.5.13 in \cite{BLM} in the complex case, which gives the real case by complexification; see also  \cite[Theorem 4.4]{RComp}), this is equivalent to $G$ acting by 
 linear surjective unital completely  isometries.

 \section{Further extension of real structure}  \label{fe}  We have seen that real structure in a complex
  operator system (resp.\ operator space) $\cS$ forces a compatible real structure in
 the $C^*$-envelope (resp.\ ternary envelope) and injective envelope.   A natural question is whether it forces real structure in the bicommutant
 of the $C^*$-envelope?   It certainly gives real structure in the bidual of the $C^*$-envelope (by Proposition 
 \ref{chco} applied to  $(\theta_{\cS})^{**}$).    
 
 We give an example showing that this is not true 
 in general.   Consider the conjugate linear period 2 involution $f^*(z) = \overline{f(\bar{z})}$ on the disk algebra and on its $C^*$-envelope $C(\Tdb)$.
 The question is if these algebras are represented nondegenerately on a Hilbert space, then does the involution extend to the bicommutant $C(\Tdb)''$ in 
 $B(H)$?   By von Neumann's bicommutant theorem this is also the weak* closure (and the strong and weak operator topology closure) of 
 the copy of $C(\Tdb)$.  
 
 Consider a countable dense set $E$ in $\Tdb$ with $E \cap \bar{E} = \emptyset$.   Note that $C(\Tdb) \subset l^\infty(E)$ isometrically
 via the canonical map $f \mapsto f_{|E}$.   This corresponds to a canonical  representation $C(\Tdb) \to B(l^2(E))$, 
 coming from the canonical representation $\lambda :  l^\infty(E)  \to B(l^2(E))$.
 We claim that $\lambda(C(\Tdb))''$ is the von Neumann algebra $\lambda(l^\infty(E))$.   Indeed $C(\Tdb)$ is weak* dense
 in $l^\infty(E)$.    To see this suppose that $g \in l^1(E)$ annihilated $C(\Tdb)$.   Fix $x \in E$ and choose a decreasing 
 sequence of positive `tent functions'
 $f_n \in C(\Tdb)$ with $f_n(x) = 1$ and $f_n = 0$ outside of an arc center $x$ of length $1/n$.   So $f_n$ converges pointwise on 
 $E$ to $\chi_{\{ x \}}$ 
 By Lebesgues theorem we have $$0 = \langle f_n , g \rangle \to \langle \chi_{\{ x \}} , g \rangle = g(x) .$$
 Thus $g = 0$ and $C(\Tdb)$ is weak* dense in  $l^\infty(E)$, and $\lambda(C(\Tdb))'' = \lambda(l^\infty(E))$ as claimed.
 
There is no  conjugate linear  $*$-automorphism of $l^\infty(E)$  extending the given involution on $C(\Tdb)$.  Indeed such a $*$-automorphism
 would be given by a period 2 permutation (bijection)
$\alpha$ of $E$.  This yields  $$\overline{f(\bar{z})} = \overline{f(\alpha(z))} , \qquad f \in C(\Tdb),$$ which 
 implies the contradiction that $\alpha(z) \in E \cap \bar{E}$.

Phrasing this in different language: let $N$ be the real $C^*$-algebra of fixed points of the given involution on $C(\Tdb)$. 
Then $C(\Tdb) = N_c$, but $(\bar{N}^{w*})_c \neq \overline{N_c}^{w*}$.  Indeed $\overline{N_c}^{w*} = l^\infty(E)$ has 
no  conjugate linear  $*$-automorphism extending the involution on $N_c = C(\Tdb)$.   In this example $N  = \{ f \in 
C(\Tdb) : f(z) = \overline{f(\bar{z})} , z \in \Tdb \}$ is real $*$-isomorphic to 
the continuous complex valued functions on the upper semicircle that are real at the endpoints $\pm1$.   
 This example also shows that one cannot extend the real structure to every injective $C^*$-algebra containing $C^*_e(A)$.  
   For $l^\infty(E)$ is an  injective $C^*$-algebra containing $C(\Tdb)$.  
   
On the other hand, note that the given involution on $C(\Tdb)$ does extend to the universal von Neumann algebra 
$C(\Tdb)^{**}$ (by e.g.\ \cite[Lemma 2.12]{BWinv}), and to $L^\infty(\Tdb)$.  The latter may be viewed as the GNS representation of the Lebesgue integral 
on $C(\Tdb)$.   

 More generally, we have:
 
  \begin{proposition} \label{tgns} Suppose that we have an involution $\dagger$ on a complex $C^*$-algebra $A$ given 
 by a conjugate linear completely isometric period 2 $*$-automorphism $\theta : A \to A$ (so that $\theta(a) = a^\dagger$, and $A$ is the complexification of a real $C^*$-algebra $B$).
 If $\tau$ is   a faithful $\dagger$-preserving state  on  $A$, 
and if $\tau$ has  GNS representation $\pi_\tau$ on $H_\tau$,  
then the involution $\theta$ extends to the von Neumann algebra   $\pi_\tau(A)''$, and indeed further to 
$B(H_\tau)$. 
 \end{proposition}

\begin{proof}   This may be seen  by first defining a complex linear   
involution $c_\tau$ on the GNS Hilbert space $H_\tau$ by $c_\tau (a) = \theta(a)$ for $a \in A$.
Note that $$\|  c_\tau (a) \|^2_{H_\tau} = \tau(\theta(a^*) \theta(a)) =  \tau(a^* a) = \|  a \|^2_{H_\tau}, \qquad a \in A .$$
So $c_\tau$ extends to a selfadjoint unitary (a symmetry) $U$ on $H_\tau$.
Then $T \mapsto U T U$ is a conjugate linear  weak* continuous automorphism   
on $B(H_\tau)$ which extends the canonical  involution on $\pi_\tau(A)$.   
Indeed $$U \pi_\tau(a) U (b) = U \pi_\tau(a)  \theta(b) = U a \theta(b) = \theta(a \theta(b)) = a^{\dagger} b  , \qquad a, b \in A, $$ 
So $U \pi_\tau(a) U (b) = \pi_\tau(a^{\dagger})$ since $A$ is dense in $H_\tau$.  
By continuity and density the restriction of this automorphism to $\pi_\tau(A)'' = \overline{\pi_\tau(A)}^{w*}$ is a conjugate linear completely isometric period 2 $*$-automorphism there. 
\end{proof} 

Thus if $A$ is the complexification of a real $C^*$-algebra $B$ as above then we conclude that   $\pi_\tau(B_c)''$ is the complexification of $\pi_\tau(B)''$.

 \end{document}